\documentclass[10pt]{amsart}
\usepackage{amsfonts,epsfig,fancyhdr}
\usepackage{amsmath,amssymb,amsthm, fancyhdr,latexsym,graphicx,graphics}
\usepackage{multicol,picinpar}
\usepackage{pdfsync}
\usepackage{url,hyperref}
\usepackage{euscript}
\usepackage{wrapfig}
\usepackage{caption}
\usepackage{tcolorbox}
 \thanks{ A. J. Di Scala is member of CrypTO, GNSAGA of INdAM and of DISMA Dipartimento di Eccellenza MIUR 2018-2022; C. E. Olmos was supported by Famaf-UNC and CIEM-Conicet; F. Vittone was supported by UNR and Conicet.}

\numberwithin{equation}{section}
\def \R{\mathbb{R}}

\newtheorem{theorem}{Theorem}[section]
\newtheorem{cor}[theorem]{Corollary}

\newtheorem{lemma}[theorem]{Lemma}
\newtheorem{rem}[theorem]{Remark}
\newtheorem{remark}[theorem]{Remark}

\newtheorem{prop}[theorem]{Proposition}
\newtheorem{thm}[theorem]{Theorem}

\usepackage{url,hyperref}

\usepackage{enumerate}
\hypersetup{}

\usepackage{euscript,color}

\def\red{\color[rgb]{1,0,0}}

\definecolor{red}{rgb}{1,0,0}

\def\R{\mathbb{R}}

\begin{document}

\title[The structure of homogeneous Riemannian manifolds with nullity]
{The structure of homogeneous Riemannian manifolds with nullity}

\author[A. J. Di Scala]{Antonio J. Di Scala}
\author[C. E. Olmos]{Carlos E. Olmos}
\author[F. Vittone]{Francisco Vittone}

%\subjclass{Primary 53C29, 53C40}
%\keywords{normal holonomy, CR-submanifolds, normal connection, s-representations.}
%%\thanks{XXXXXXXXXXXXXXXXXXXXXXXXXXXXXXXXXXXXX. }

\date{\today}
\begin{abstract}
	We find new conditions that the existence of nullity of the curvature tensor of an irreducible homogeneous space $M=G/H$ imposes on the Lie algebra $\mathfrak g$ of $G$ and on the Lie algebra $\tilde{\mathfrak g}$ of the full isometry group of $M$.  Namely, we prove that there exists a transvection of $M$ in the direction of any element of the nullity, possibly by enlarging the presentation group $G$. Moreover, we prove that these transvections generate an abelian ideal of $\tilde{\mathfrak g}$. These results constitute a substantial improvement on the structure theory developed in \cite{DOV}. In addition we construct examples of homogeneous Riemannian spaces with non-trivial nullity, where $G$ is a non-solvable group, answering a natural open question. Such examples admit (locally homogeneous) compact quotients.
    In the case of co-nullity $3$ we give an explicit description of the isometry group of any homogeneous
    locally irreducible Riemannian manifold with nullity.
\end{abstract}

\maketitle

\section {Introduction} \label {ii}
 In a previous paper \cite{DOV} we developed a general theory for the structure of irreducible homogeneous spaces $M=G/H$ in relation to the nullity distribution $\nu$ associated to the curvature tensor $R$.
The nullity subspace of $M$ at $p$ is defined as
 $$\nu_p=\{v\in T_pM\,:\,R_{\,\cdot\,,\,\cdot }v=0\}.$$
 Since $M$ is homogeneous, all these subspaces have the same dimension, called the \textsl{index of nullity} of $M$. The assignment $p\mapsto \nu_p$ defines an  autoparallel distribution, with flat totally geodesic integral manifolds, called the \textit{nullity distribution} or simply the \textsl{nullity} of $M$. The manifold  $M$ is said to have \textsl{trivial nullity} if either $\nu=\{0\}$ or $\nu$ is the tangent space of a local de Rham flat factor of $M$ (i.e., $\nu$ is a parallel distribution with respect to the Levi-Civita connection of $M$).

 The theory developed in \cite{DOV} allowed us to prove some interesting results. On the one hand, we showed that if $M=G/H$ is a simply connected homogeneous Riemannian manifold  without Euclidean de Rham factor such that the Lie algebra $\mathfrak g$ of $G$ is reductive (in particular, if $M$ is compact) or $\mathfrak g$ is $2$-step nilpotent, then $\nu$ must be trivial (cf. \cite[Proposition C]{DOV}).
On the other hand, we were able to construct the first known examples of irreducible homogeneous spaces with non-trivial nullity. All these examples are solvmanifolds which do not admit compact quotients. Moreover, we showed in part (4) of \cite [Theorem A]{DOV}  that if the \textsl{co-nullity} (i.e. the co-dimension of the nullity $\nu$) is $k=3$, then the presentation group $G$ must be solvable.  However, as we pointed out in the Introduction of \cite{DOV}, we did not know of any example of an irreducible simply connected homogeneous Riemannian manifold $M=G/H$ with positive index of nullity such that either $H$ is non-trivial or $G$ is non-solvable.
In this paper we answer the  second question by constructing examples of Riemannian homogeneous spaces with non-trivial nullity and  non-solvable presentation group. Such examples admit compact (locally homogeneous) quotients (see Section \ref{nonsolvable}).

The main goal of the paper is to revisit the structure theory developed in \cite{DOV} and substantially improve it. Let $M = G/K$ be a locally irreducible homogeneous Riemannian manifold with non-trivial nullity $\nu$. One can consider some natural distributions associated to $\nu$. Namely,
	$\nu ^{(1)}, \nu ^{(2)}$, $\hat \nu$ and $\mathcal U$.
	
The distributions $\nu^{(1)}$ and $\nu^{(2)}$ are called the \textsl{osculating distributions} of order $1$ and $2$ of $\nu$, respectively. If one puts $\nu^{(0)}=\nu$, then $\nu^{(i+1)}$ is obtained by adding to $\nu^{(i)}$ the covariant derivative, in any direction, of fields that lie in $\nu^{(i)}$ ($i=0,1$). The distribution  $\nu^{(1)}$ can, in turn, be decomposed as $$\nu^{(1)}=\nu+\hat{\nu}$$
where $\hat{\nu}_p$ is the linear span of $\{\nabla_{\nu_p}Z\}$ with $Z\in \mathcal{K}^G(M)\simeq \mathfrak g$, the space of Killing fields of $M$ whose flow belong to $G$ (see equation \ref{nupic}). The distribution $\hat{\nu}$ is called the \textsl{adapted distribution} of $\nu$.

Finally, the distribution $\mathcal{U}$, called the \textsl{bounded distribution}. Namely, $\mathcal{U}_p$ consists of the directions of the Killing fields in $\mathcal{K}^G(M)$ whose covariant derivatives   along the  leaf of nullity  $N(p)$   lie in  $\nu$ (see equation(\ref{bounded})).

All these distributions are $G$-invariant
 and  verify the following inclusions  (see \cite [Theorem A (1)] {DOV}):
$$\{0\} \subsetneq \nu \subsetneq \nu ^{(1)}\subsetneq \nu ^{(2)}\subset \mathcal U
\subsetneq TM.$$
This imposes a first important restriction: if $M$ has non-trivial nullity, then the co-nullity must be at least $3$.

Another important consequence of a non-trivial nullity  is the existence of the so-called \textsl{adapted transvections}. These are transvections induced by $G$ (i.e. elements of $\mathcal{K}^G(M)$ such that $(\nabla X)_p=0$) whose initial condition lie in $\hat{\nu}$. In \cite [Theorem A (2)]{DOV} it was proven that for each $v\in \hat{\nu}_p$ there exist an adapted transvection $Y$ with $Y_p=v$ such that the Jacobi operator $R_{\,\cdot\,,\,v}v= 0$,  $[Y,[Y,\mathcal{K}(M)]]=0$ and $Y$ does not belong to the center of $\mathcal{K}^G(M)$ (where $\mathcal{K}(M)$ denotes the space of Killing fields of $M$). The existence of  an adapted transvection $Y$ of order two, i.e. $ad_Y^2=0$.  plays a crucial role in the construction of the examples of \cite [Section9]{DOV}).

In this paper we further examine the restrictions that the existence  of nullity imposes on the Lie algebra $\mathcal{K}(M)$. In particular, we prove that there exists a transvection, that might not belong to $\mathcal K^G(M)$,  in the direction of any element of $\nu$. Moreover, we prove that these transvections generate an abelian ideal $\mathfrak a$ of $\mathcal K(M)$. Our main structure result is the following:

\begin {thm} \label{dakl} Let $M=G/H$ be a simply connected  irreducible homogeneous Riemannian manifold with a non-trivial nullity distribution $\nu$, let $\mathfrak g =\text{Lie}(G)\simeq \mathcal K^G (M)$ and $\tilde{\mathfrak g}= \text {Lie}(I(M))\simeq
\mathcal K(M)$. Then
\begin {itemize}

\item [1)] For any $p\in M$, $v\in \nu _p$ there exists a transvection with initial condition $v$ (possibly, not in $\mathfrak g$).

\item [2)] Let $\tilde{\mathfrak {tr}}^p_0\subset \tilde{\mathfrak g}$ be the vector space of transvections at $p$ with initial condition in $\nu _p$. Then
$\mathfrak g' : = \tilde{\mathfrak {tr}}^p_0+ \mathfrak g$ is a Lie subalgebra of $\tilde {\mathfrak g}$ which does not depend on $p\in M$. Moreover, $\mathfrak g$
is an ideal of $\mathfrak g'$.

\item [3)] The set $\tilde{\mathfrak {tr}}^p_0$ generates an abelian ideal $\mathfrak a$ in $\tilde{\mathfrak g}$. Moreover, $\mathfrak a$  does not depend on $p\in M$ and coincides with the ideal  generated by $\tilde{\mathfrak {tr}}^p_0$ in $\mathfrak g'$.

\item [4)] The integrable distribution $\mathcal D$ of $M$, given by  $\mathcal D_q= \mathfrak a.q$,   contains $\nu ^{(2)}$. Moreover,
$\nu$ and $\nu ^{(1)}$ are parallel along any orbit
$\bar A\cdot q$, where $\bar A$ is the closure in $I(M)$  of the Lie group associated to $\mathfrak a$.
\end{itemize}

\end{thm}

Part (2) and (3) of the above theorem explains why the adapted transvections  of \cite[ Theorem A (2)]{DOV} have order two. In fact, any adapted transvection can be obtained as the Lie bracket of a transvection in $\tilde{\mathfrak{tr}^p_0}$ and a Killing field in $\mathcal K^G(M)$ (see Remark \ref{rem33}). Hence the subspace of adapted transvections is contained in the abelian ideal $\mathfrak a$.

Moreover, since the adapted transvections cannot lie in the center of $\mathcal{K}^G(M)$, the existence of the abelian ideal $\mathfrak a$ implies that the Lie algebra $\mathfrak g$ of $G$ cannot be reductive. This gives a conceptual much simpler proof of \cite[Proposition C]{DOV}.

The results of Theorem \ref{dakl} also allow us to improve part (4) of Theorem A of \cite{DOV}. In fact, we obtain the following explicit description of the Lie group $G$ in case of co-nullity $k=3$.
\begin{cor}\label{cor:A}
Let $M^n=G/H$ be a simply connected  irreducible homogeneous Riemannian manifold of co-nullity $3$, where $G$ is connected ($n>3$). Then $H$ is trivial and  $G= \mathbb R^{n-1}\rtimes \mathbb R$.
\end{cor}

Observe that the projection from $M$ to the quotient $\bar A\backslash M$ by the orbits of $\bar A$ is a $G$-invariant Riemannian submersion with intrinsically flat fibers.

\
Finally, we give an explicit way to construct non-solvable examples. Namely:

\begin {thm}\label{mex2}
Let $K$ be a simply connected compact simple Lie group and let
$\rho : K\to SO_n$ be an irreducible orthogonal representation. Let $\mathbb V_0$ be a non-trivial vector subspace of $\mathbb R^n$ such that $\dim (\mathbb V_0)(1 +\dim (K))< n$. Then there exists a left invariant metric $\langle \, ,\, \rangle$ on $G=\mathbb R^n \rtimes_{\rho} K$
such that $M = (G,\langle \, ,\, \rangle)$ is an irreducible Riemannian manifold and the nullity subspace $\nu_{e}$ contains $\mathbb{V}_0$ (and hence has dimension at least
$\dim (\mathbb V_0)$). In particular, the quotient $N=\mathbb Z^n\backslash M$ of $M$ by the orbits of
$\mathbb Z^n$ is a compact, locally irreducible, non-homogeneous but locally homogeneous Riemannian manifold with a non-trivial nullity distribution.
\end{thm}

Note that  for any simply connected and simple compact Lie group $K$ there exists a representation $\rho:K\to SO_n$ for  arbitrary large $n$.  The manifold $M$ given by the above theorem is homotopic to a compact simply connected Lie group and thus it admits no solvable transitive group of isometries.

Observe that, as we mentioned before, any compact locally irreducible homogeneous Riemannian manifold has a trivial nullity distribution (\cite[Proposition C]{DOV}). Then the compact manifold
$N= \mathbb Z^n\backslash M$, given by the above theorem, shows that the assumption of homogeneity can not be replaced by local homogeneity. This also answers a natural question (other examples of compact locally homogeneous spaces with non-trivial nullity were found in \cite{GG}).

\section {Preliminaries} \label {prel}
 In this section we introduce some basic facts which will be needed in the rest of the paper. For further details we refer to  \cite {DOV}.
Let $(M,\langle \, ,\, \rangle)$ be a (connected) complete Riemannian manifold with Levi-Civita connection $\nabla$ . A vector field $X$ of $M$ is called a {\it Killing field} if
$(\nabla _{\cdot} X)_p$
 is a skew-symmetric endomorphisms  of $T_pM$, for all $p\in M$.
 Such a condition, the so-called  {\it Killing equation}, is equivalent to the  fact that the  flow $\phi _t$ of $X$ is by isometries.
 The Lie algebra
$\mathrm {Lie}(I(M))$ of the group $ I(M)$ of isometries of $M$ is naturally identified with
the Lie algebra $\mathcal K (M)$ of Killing fields of $M$. Namely,
 $u\in \mathrm {Lie}(I(M))= T_e(I(M))$ induces the Killing field $\tilde u$
 defined by $\tilde u_p= \tfrac {\mathrm d\,}{\mathrm dt}_0\mathrm {Exp}
 (tu).p$ and the map $u\mapsto \tilde u$ is a linear isomorphism.
 Moreover, $[u,w]^{\tilde \ }= -[\tilde u, \tilde w]$ since
 $\tilde u$ and $\tilde w$ are naturally identified with the right invariant fields of $I(M)$ with initial conditions
 $u$ and $w$, respectively.

Let  $G$ act by isometries on $M$. The Killing fields of $M$ induced by elements in the Lie algebra of $G$ are denoted by
 $\mathcal K ^G(M)$ (briefly, the Killing fields induced by $G$).
 If the action of  $G$  on $M$ is not almost effective, then there  exist non-zero elements $z\in\mathfrak{g}$ such that  $\tilde z = 0$.

\medskip

Let $X\in \mathcal K (M)$. The initial conditions of $X$ at  $p\in M$  are given by  the pair
$$(X)^p: =(X_p, (\nabla X)_p)\in T_pM\oplus \Lambda^2(T_pM).$$
These conditions completely determine the Killing field $X$, in the sense that two Killing fields with the same initial conditions at some   point $p$  must coincide on $M$.
A Killing field is called a {\it transvection} at $p$ if its initial conditions at $p$ are $(X)^p: =(X_p, 0)$, i.e., if $(\nabla X)_p=0$.

A Killing field $X$ satisfies, besides the Killing equation,
the following identity, for all $p\in M$, $u,v\in T_pM$
\begin {equation}
\label {affK}
\ \ \ \ \ \ \ \ \ \ \ \ \nabla ^2_{u,v}X = R_{u,X_p}v
\ \ \ \ \ \ \ \ \ \ \mathit {affine \ Killing \ equation}.
\end {equation}
where $R$ is the Riemannian curvature tensor of $M$. The affine Killing equation reflects the fact that the flow of $X$ preserves the Levi-Civita connection.

From the affine Killing equation and the Bianchi identity one can determine the initial conditions at $p$ of the bracket $[X,X']$ of  two Killing fields
in terms of the initial conditions $(X)^p = (v, B)$, $(X')^p = (v', B')$. Namely,
\begin{equation} \label{initb}
([X,X'])^p = (B'(v)-B(v'), R_{v, v'} - [B,B']).
\end{equation}

This equation gives a useful formula for computing the curvature in terms of the  Killing fields $X$ and $Y$. Namely:
 \begin {equation} \label {feqR}
R_{X_p,Y_p} = (\nabla [X,Y])_p + [(\nabla X)_p, (\nabla Y)_p]
 \end {equation}

 The so-called Koszul formula gives the Levi-Civita connection  in terms of
brackets of vector fields and scalar products.
From the fact that the Lie derivative of the metric tensor along any Killing vector field is zero one has    the following useful formula (see e.g. (3.4)  of \cite {ORT}):
\begin {equation} \label{feq}
2\langle \nabla _XY , Z\rangle = \langle [X,Y],Z \rangle
+ \langle [X,Z],Y \rangle + \langle [Y, Z], X \rangle.
 \end {equation}

 Let $X$ be a Killing field and let $\phi _t (q)$ be its associated flow.
 Let $p\in M$ and let $c(t)= \phi _t (p)$ be the integral curve of $X$ by $p$. Let $\tau _t$ denote the parallel transport along $c(t)$, form $0$ to $t$. Then
 $\tau _{t}^{-1}\circ \mathrm {d}\phi _t : T_pM \to T_pM$ is a $1$-parameter subgroup of
 linear isometries. Moreover, (see e.g. Remark 2.3 of \cite {OS})
 \begin {equation} \label {etnabla}
 \tau _{t}^{-1}\circ \mathrm {d}_p\phi _t  =\mathrm{e} ^ {t(\nabla X)_p}
 \end {equation}
 or equivalently
 \begin {equation} \label {etnabla2}
 \tau _{t}= \mathrm {d}_p  \phi _t  \circ \mathrm{e} ^ {-t(\nabla X)_p}.
 \end {equation}

 \begin {remark} \label {accm} {\rm (cf. \cite [Remark 2.3] {DOV}).}
 Let $X\in \mathcal K(M)$ and
 let $T\subset SO (T_pM)$ be the torus given by the closure of $\{\mathrm{e} ^ {t(\nabla X)_p}: t \in \mathbb R\}$.
 Then for any given $d \in T$
  there is a sequence of real numbers  $\{t_n\}_{n\in \mathbb N}$ which tends to $+ \infty$ and such that
 $\tau _{t_n}^{-1}\circ \mathrm {d}_p\phi _{t_n}  =\mathrm{e} ^ {t_n(\nabla X)_p}$
 tends to $d$. In particular,
 for $d= \mathrm{e} ^ {s_0(\nabla X)_p}$, where $s_0\in \mathbb R$ is arbitrary.
\end {remark}

 The nullity of the curvature tensor $R$ at $p\in M$ is  the subspace $\nu_p$ of $T_pM$ defined as
$$\nu _p := \{v \in T_pM: R_{v, x} = 0, \forall x \in T_pM \}$$
or, equivalently, from the identities of the curvature tensor,
$$\ \ \ \ \ \nu _p := \{v \in T_pM: R_{x, y}v = 0, \forall x, y \in T_pM\}.$$

The map $\nu: p\mapsto\nu_p$ defines an autoparallel  distribution in the open and dense subset $\Omega$ of
$M$ where the dimension $\dim (\nu _q)$ is locally constant,  called the \textit{nullity distribution} of $M$. If $M$ is homogeneous, then $\nu$ is a distribution in $M$ and one has:
\begin{theorem}[cf. \cite{DOV} Theorem C]\label{compacto}  If $M$ is a homogeneous compact Riemannian manifold with no  Euclidean local de Rham factor, then its nullity distribution is trivial.
\end{theorem}
We recall Lemma 2.8 of \cite {DOV}:
\begin {lemma}\label {ABC} Let $M$ be a Riemannian manifold.
Let  $\gamma _v (t) $ be  a geodesic  with $\gamma_v(0)=p$, $\gamma_v'(0)=v$ and such that $\gamma_v'(t)$ belongs to the nullity subspace $\nu _{\gamma _v (t)}$  for every $t$ (if $\nu$ is a distribution, then this is equivalent to the fact that $v\in \nu _{p}$).  Denote by   $\tau_t$  the parallel transport along $\gamma _v(t)$ from $0$ to $t$.  If  $Z$ is an arbitrary Killing field on $M$, then

\begin {enumerate}[(i)]
\item  $ Z_{\gamma _v (t)} =  \tau _t (Z_p) + t \tau _t(\nabla _v Z)$.

\item  $\nabla _{\gamma_v '(t)} (\nabla Z) =0$, i.e.,  $\nabla Z$ is parallel along $\gamma _v (t)$, or equivalently
$$(\nabla Z)_{\gamma _v (t)}= \tau _t ((\nabla Z)_p) : = \tau _t
\circ (\nabla Z)_p \circ \tau _t^{-1}.$$
\end {enumerate}
\end {lemma}

Let $M=G/H$ be a simply connected  homogeneous Riemannian manifold without Euclidean de Rham factor and assume that its nullity distribution $\nu$ is non-trivial. Observe that $\nu$ is not a parallel distribution, otherwise $M$ would split off a Euclidean factor. In addition, in this case, a non-trivial Killing field cannot be always tangent to the nullity distribution  (\cite [Propositon 3.19]{DOV}).

 For a $G$-invariant distribution $\mathcal{D}$ on $M$, consider the \textsl{osculating distribution} $\mathcal{D}^{(1)}$ of $\mathcal{D}$ defined as
	$$\mathcal{D}^{(1)}_q:=\mathcal{D}_q+ \mathrm{span}\,  \{ \nabla _w X : X \in C^\infty (\mathcal{D}), w\in T_qM\}$$
where $C^{\infty}(\mathcal{D})$ denotes the tangent fields of $M$ that lie in $\mathcal{D}$. It is not difficult to see that $\mathcal{D}^{(1)}$ is a $G$-invariant distribution of $M$ such that $\mathcal{D}\subset \mathcal{D}^{(1)}$. Moreover, if $\mathcal{D}$ is not parallel, then $\mathcal{D}^{(1)}$ properly contains  $\mathcal{D}$.
On has that (cf. \cite[Lemma 3.1]{DOV})
\begin{equation}\label{firstosculating}\mathcal{D}^{(1)}=\mathcal{D}+\hat{\mathcal{D}}\end{equation}
where
$$\hat{\mathcal{D}}_q:=\{\nabla_v Z\,:\, Z\in \mathcal{K}^G(M),\, v\in \mathcal{D}_q\}.
$$
$\hat{\mathcal{D}}$ is called the \textsl{adapted} distribution of $\mathcal{D}$. The adapted distribution might depend on the presentation group $G$, but the higher order osculating distributions $\mathcal{D}^{(k)}:=(\mathcal{D}^{(k-1)})^{(1)}$ (where $\mathcal{D}^{(0)}=\mathcal{D})$ do not depend on $G$.

Let $\nu^{(1)}$  be the osculating distribution of the nullity distribution and set $\nu^{(2)}:=(\nu^{(1)})^{(1)}$, called the second order osculating distribution of $\nu$. Both $\nu^{(1)}$ and $\nu^{(2)}$ are $G$-invariant distributions.  Since $\nu$ is not parallel, $\nu$ is properly contained in $\nu^{(1)}$. So one has that
\begin{equation}\label{nupic}
	\nu_q ^{(1)}= \nu _q + \hat{\nu}_q=\nu_q+ \mathrm {span}\,  \{\nabla _v Z: Z\in \mathcal K ^G (M) , v\in \nu _q\}
\end{equation}
and  one can describe $\nu^{(2)}$ as
 \begin{equation}\label{nupic4}\nu ^{(2)} = \nu ^{(1)} + \hat {\nu ^{(1)}}=\nu + \hat{\nu} + \hat{\hat {\nu}}.
\end{equation}
 where $\hat{\hat {\nu}}= \mathrm {span}\,  \{\nabla _v Z: Z\in \mathcal K ^G (M) , v\in \hat {\nu} _q \}$ (cf. \cite[Formula 5.0.2]{DOV}).

\medskip
The osculating distributions of $\nu$ have the following properties:

\begin{lemma}[\cite{DOV} Theorem A and Corollary 3.13]\label{propsosculating} Let $M$ be a Riemannian homogeneous manifold which does not split off a local Euclidean de Rham factor and with non-trivial nullity distribution $\nu$. Then:
	
\begin{enumerate}
	\item $\nu$ and $\nu ^{(1)}$ are autoparallel distributions and $ \nu \subsetneq \nu ^{(1)}\subsetneq \nu ^{(2)}\subsetneq TM$. In particular,  the co-dimension of $\nu$ is at least $3$.
\item For any $v\in \hat {\nu}_q$ there exists a  transvection $Y$ of $M$ which belongs to $\mathcal K^G(M)$ and such that $Y_p=v$ (such a transvection is called an \textsl{ adapted transvection}). Moreover, the  Jacobi operator $R_{.,\, v}v$ is null,  $$[Y, [Y,\mathcal K(M)]] = \{0\}$$ and $[Y,\mathcal K^G(M)] \neq  \{0\}$, i.e $Y$ does not belong to the center of $\mathcal{K}^G(M)$.
\end{enumerate}
\end{lemma}

Observe that if  $X$ is a transvection at $p$, since $\nabla$ is torsion free,
\begin{equation}\label {5565} \nabla _{X_p}Z  = [X, Z]_p \, .\end{equation}
Hence, from property (2) of Lemma \ref{propsosculating} one gets that
\begin{equation}\label {666}
	\hat{\hat {\nu}}_p = \{{[X, Z]_p}: Z \in \mathcal {K}^G(M), X \text { is an adapted transvection at } p\}. \end{equation}

\section {The ideal {$\mathfrak u_{00}$}} \label{sec3}
We keep the notation of Section \ref{prel}.
Let $M=G/H$ be a simply connected Riemannian homogeneous manifold without Euclidean de Rham factor and   assume that the nullity distribution $\nu$ of $M$
is non-trivial.

The \textit{bounded algebra of $M$ at $p$} is given by
\begin{equation}\label{bounded}\mathfrak u^p : = \{Z\in \mathcal K ^G(M): \nabla _ {\nu _p}Z \subset
\nu _p\}. \end{equation}
From formula (\ref {initb}) one obtains that $\mathfrak u^p$ is a Lie subalgebra of $\mathcal K^G(M)$ (see \cite [Section 5] {DOV})  which verifies that for every $g\in G$,  $$\mathfrak u^{gp}= \mathrm {Ad}_g\mathfrak u^{p}.$$
So $\mathcal U_p=  \mathfrak u^p.p$ defines an integrable distribution. Namely, the integral manifold of $\mathcal U$ by $p$ is the orbit by $p$ of the Lie group associated to  $\mathfrak u^p$.
The distribution $\mathcal U$ is the so-called {\it bounded distribution}. It  contains the first two osculating distributions (\cite[Theorem A]{DOV}). More precisely, on has that
	\begin{equation}\label{osculatingandbounded}
		\nu \subsetneq \nu ^{(1)}\subsetneq \nu ^{(2)}\subset \mathcal U\subsetneq TM.
	\end{equation}

Observe that any transvection at $p$ belongs to $\mathfrak u^p$.  In particular,
\begin{equation} \label{eq:trp1}
\mathfrak {tr}^p := \{
    X\in \mathcal K^G(M): X \textrm{ is a transvection at $p$ with } X_p\in \nu ^{(1)}_p\} \subset \mathfrak u^p.
\end{equation}

Moreover, by   \cite[Remark 3.10]{DOV} one has that
\begin{equation}\label{trbau}
	R_{\, \mathfrak {tr}^p,\mathfrak u^p}=0
\end{equation}
Then, from (\ref{initb}),  it follows that
$\mathfrak {tr}^p$ is an abelian Lie subalgebra of $\mathfrak u ^p$.

 Observe that since $\hat{\nu}_p\subset \nu_p^{(1)}$, from item (2) of Lemma \ref{propsosculating} one gets that for each $u,v\in \hat{\nu}_p$ there exist transvections $X,Y\in \mathfrak{tr}^p$ such that $X_p=u$, $Y_p=v$. So by (\ref{trbau}) $R_{u,v}=R_{X_p,Y_p}=0$, i.e.,
	\begin{equation} \label{adtrbau}
R_{\hat{\nu}_p,\hat{\nu}_p}=0.
	\end{equation}

\begin{remark}\label{rem890} From  (\ref{nupic}) and  Lemma \ref{propsosculating}, any  $X\in \mathfrak {tr}^p$ is the sum of an adapted transvection at $p$ (i.e., a transvection whose value al $p$ belongs to $\hat{\nu}_p$) and a transvection tangent to the nullity at $p$. But we do not know whether there exists a transvection in any direction of $\nu_p$. We will prove this fact in the next section (but such a transvection might not belong to $\mathcal K ^G(M)$).
\end{remark}

\begin{rem}\label {rem-uno}
From Lemma \ref{propsosculating} one has that $R_{\, \cdot  ,X_p}X_p =0$, for all $X\in \mathfrak {tr} ^p$. Polarizing this formula one gets that
$R_{\, \cdot  ,X_p}Y_p + R_{\, \cdot ,Y_p}X_p = 0$ for all $X, Y \in  \mathfrak {tr} ^p$. Note that from  Bianchi identity
$R_{\, \cdot   ,X_p}Y_p - R_{\, \cdot  ,Y_p}X_p = -R_{X_p,Y_p}\cdot = 0$ (see (\ref {trbau})).
Then,  for all $X,Y\in \mathfrak {tr} ^p$, $$R_{\, \cdot  ,X_p}Y_p =0.$$
\end{rem}
Let
\begin{equation}\label {u_0}\mathfrak u^p_{0} : = \{Z\in \mathcal K ^G(M): \nabla _ {\nu _p}Z = 0  \}
\end{equation}
and
\begin{equation}\label {u_00}
\mathfrak u^p_{00} : = \{Z\in \mathcal K ^G(M):
\nabla _ {\nu _p}Z = 0
\text { and } \nabla _ {\hat {\nu} _p}Z = 0 \}.
\end {equation}
Equivalently, by  (\ref {nupic}),
\begin{equation}\label {u_00'}
    \mathfrak u^p_{00} : = \{Z\in \mathcal K ^G(M):
\nabla _ {(\nu^{(1)}) _p}Z = 0 \}
\end{equation}

Clearly,
\begin{equation}\label{001}\mathfrak {tr} ^p \subset \mathfrak u^p_{00}\subset \mathfrak u ^p_0\subset \mathfrak u^p.
\end{equation}

\begin {lemma}\label{abcdef}
If $U\in \mathfrak u ^p_0$, then $\nabla _{\hat {\nu}_p}U \subset \hat {\nu}_p$.
\end{lemma}
\begin{proof} Let $(U)^p=(u,B)$ be the initial conditions of $U$. Let $Z\in \mathcal K ^G(M)$ be arbitrary with initial conditions
$(Z)^p = (z, B')$  and put $W = [U,Z]$. Let $v\in \nu _p$ be arbitrary.
Observe that, from the definition of $\hat {\nu}_p$,  $ \nabla _vW\in \hat {\nu}_p$.
From (\ref {initb}) the initial conditions of $W$ at $p$ are
$$(W)^p = (B'(u) - B(z), R_{u,z} - (BB' -B'B)).$$
Since $v\in \nu_p$, one has that  $B(v) = 0$ and $R_{u,z}(v)=0$. Hence
$$\nabla _vW=- B(B'(v))\in \hat {\nu}_p.$$
But the set $\{B'(v)\}$, for $Z\in \mathcal K^G(M)$ and $v\in \nu _p$ arbitrary, spans $\hat {\nu}_p$.
Then $$B(\hat {\nu}_p)\subset \hat {\nu}_p\, .$$
\end{proof}

\begin{thm}\label{idealmain} $\mathfrak u^p_{00}$ is an ideal of $\mathcal K^G(M)$ and does not depend on $p\in M$ (and so in the next we will refer to it as
$\mathfrak u_{00}$). Moreover, for any $p\in M$,  $\mathfrak {tr}^p$  is contained  in the center
$\mathfrak z ( \mathfrak u_{00})$ of  $\mathfrak u_{00}$.

\end{thm}
\begin{proof}
Let $Z\in \mathcal K ^G(M)$ be arbitrary with initial conditions
$(Z)^p = (z, B')$ and let $U\in\mathfrak u^p_{00}$,  with initial conditions $(U)^p=(u,B)$. Set $W = [U,Z]$. Let us see first that $W\in \mathfrak u^p_0$.

From (\ref{initb}),
the initial conditions of $W$ at $p$ are
$$(W)^p = (B'(u) - B(z), R_{u,z} - (BB' -B'B)).$$
Since $U\in \mathfrak u^p_{00}$ one has that $B(\nu_p)=B(\hat{\nu}_p)=0$.  On the other hand, from the definition of $\hat{\nu}$ on gets that $B'(\nu_p)\subset \hat{\nu}_p$, hence $B(B'(\nu_p))=0$.

So, if  $v\in \nu _p$ is arbitrary, then$$\nabla _vW =  R_{u,z}v - B(B'(v)) +B'(B(v))= -B(B'(v)) = 0.$$
Then $W\in \mathfrak u ^p_0$. So we need to prove that $\nabla_{\hat{\nu}_p} W=0$.

Let now $r\in \hat {\nu}_p$. From Lemma \ref {abcdef} one has that
$\nabla _rW\in \hat {\nu}_p$. So, in order to show that $\nabla _rW=0$, we only have to show
that the projection of $\nabla _rW$ to $\hat {\nu}_p$ is zero. Let $s\in \hat {\nu}_p$ be arbitrary. Then
$$ \langle \nabla _rW,s\rangle = \langle R_{u,z}r,s\rangle -
\langle B(B'(r)),s\rangle  +\langle B'(B(r)), s\rangle.$$
Since $r, s\in \hat{\nu}_p$, by (\ref{adtrbau}) $R_{r,s}=0$.  Hence $\langle R_{u,z}r,s\rangle =\langle R_{r,s}u,z\rangle=0.$

On the other hand, recall that $B (\hat{\nu}_p)=0$ hence  $B(r)=B(s)=0$ and then
$$ \langle \nabla _rW,s\rangle =  -
\langle B(B'(r)),s\rangle =\langle B'(r),B(s)\rangle=0.$$
Then $W\in \mathfrak u ^p_{00}$ and thus  $\mathfrak u ^p_{00}$ is an ideal of
$\mathcal K^G(M)$.

Being $\mathfrak u^p_{00}$ an ideal it does not depend on $p\in M$, since
$\mathfrak u^{gp}_{00} = \mathrm {Ad}(g) \mathfrak u^p_{00} $.

By making use of (\ref {initb}), (\ref{trbau}) and (\ref{001}) one obtains that
$\mathfrak {tr}^p \subset \mathfrak z ( \mathfrak u^p_{00})$ for all $p\in M$.
\end{proof}

\begin{remark}\label{remcen} Observe that the center of an ideal of a Lie algebra $\mathfrak g$ is an abelian ideal of
$\mathfrak g$. Then $\mathfrak z ( \mathfrak u_{00})$ is an abelian ideal of $\mathcal K^G(M)$.
\end{remark}

\section {Transvections in the directions of the nullity}

 All throughout this section   $M=G/H$ will be a homogeneous Riemannian manifolds with nullity distribution $\nu$ and curvature tensor $R$.

\begin{lemma} \label {util}  Let $p\in M$, $v\in \nu _p$. Then $\nabla _v R=0$.
%\begin{enumerate}[i)]
%\item $R_{\nabla _vY, Z_p} + R_{ Y_p, \nabla _vZ}= 0$, for all $Y, Z\in \mathcal K^G(M)$.
%\item $\nabla _v R=0$.
%\end{enumerate}
\end{lemma}
\begin{proof}
 Since the first terms of both the equations in \textit{i)} and \textit{ii)} are linear in $v$, it is clear that we only need to prove them for $v$ in a set of generators of $\nu_p$. Now each leaf of $\nu$ is a homogeneous, flat totally geodesic submanifold of $M$ (see \cite[Paragraph 3.2]{DOV}). Hence $\nu_p$ is generated by the direction of homogeneous geodesics tangent to $\nu$, which are given by the flow of some Killing field in $\mathcal{K}^G(M)$ (cf. \cite[Paragraph 2.5]{DOV}).
	
	 So let $v\in\nu_p$ be such that there exist a homogeneous geodesic $\gamma_v(t)$  tangent to $\nu$ with $\gamma_v(0)=p$ and $\gamma'_v(0)=v$ and let $X\in \mathcal{K}^G(M)$ be such that $$\gamma_v(t)=\phi_t(p)$$ where $\phi_t$ is the flow associated to $X$.

Let $\tau_t$ be the parallel transport along $\gamma_v(t)$ from $0$ to $t$.
By Lemma \ref{ABC}, $$Z_{\gamma _v (t)} =  \tau _t (Z_p) + t \tau _t(\nabla _v Z)$$ and $\nabla _{\gamma_v'(t)} (\nabla Z)  =0$ and the same is true if one replaces  $Z$ by $Y$. From (\ref {initb}) one has that
$$(\nabla [Y,Z])_{\gamma _v(t)} =  R_{Y_{\gamma _v(t)}, Z_{\gamma _v(t)}} - [(\nabla Y)_{\gamma _v(t)} ,
(\nabla Z)_{\gamma _v(t)} ]. $$
Observe  that $C(t) = [(\nabla Y)_{\gamma _v(t)} ,
(\nabla Z)_{\gamma _v(t)} ]$ is a parallel skew symmetric $(1,1)$ tensor along $\gamma_v(t)$
since both terms of the bracket are so. Then $$C(t)= \tau _t(C(0)) = \tau _t\circ C(0)\circ \tau _t^{-1}.$$
Set $u= Y_p$, $u'= \nabla _vY$, $w= Z_p$, $w'=\nabla _vZ$.
Then,
\begin{equation} \label {555}
\begin{split}
(\nabla [Y,Z])_{\gamma _v(t)} = &\
 R_{\tau _t (u),\tau _t (w)}
+ t (R_{\tau _t (u),\tau _t (w')} + R_{\tau _t (u'), \tau _t (w)})
\\ &+ t^2 R_{\tau _t (u'), \tau _t (w')} {\red -} C(t)
\end{split}
\end{equation}

By (\ref{etnabla2}) one has that
\begin{equation} \label {5554} \tau _t   = \mathrm {d} \phi _t \circ g_t \, ,\end{equation}
 where $g_t: = \mathrm {e}^{-t(\nabla X)_p}$ is a one parameter group of linear isometries of $T_pM$. So, since $\phi _t$ preserves the curvature,

\begin{equation} \label {556}
\begin{split}
& (\nabla [Y,Z])_{\gamma _v(t)} -
 R_{\tau _t (u),\tau _t (w)} =  \\
 & \mathrm {d} \phi _t  \bigg( t (R_{g_t(u),g_t(w')} +  R_{g_t(u'),g_t(w)}) +
 t^2 R_{g_t(u'), g_t(w')}\bigg) {\red -} \tau _t(C(0)))  \\
& = \mathrm {d} \phi _t \circ \bigg( t (R_{g_t(u),g_t(w')} +  R_{g_t(u'),g_t(w)}) +
 t^2 R_{g_t(u'), g_t(w')}\bigg) \circ  \mathrm {d} \phi _t^{-1} {\red -} \\
 &  \tau _t \circ  C(0)\circ \tau _t^{-1}
\end{split}
\end{equation}

Observe that, since $M$ is homogeneous,  $\Vert R_{x,y}\Vert $ is bounded if
$\Vert x\Vert $ and $\Vert y\Vert $ are bounded.
In addition, recall that the covariant derivative  of any Killing field is parallel along the nullity (cf. Lemma \ref{ABC}). Hence $(\nabla [Y,Z])_{\gamma _v(t)}$
is parallel and thus the norm of the left hand side of
(\ref{556}) is bounded.

On the other hand, $\Vert  \tau _t \circ  C(0)\circ \tau _t^{-1}\Vert$ is also bounded and then
\begin {equation}\label {4444}\alpha (t) = \Vert  t (R_{g_t(u),g_t(w')} +  R_{g_t(u'),g_t(w)}) +
 t^2 R_{g_t(u'), g_t(w')}\Vert
\end{equation}
must be bounded.

Fix an arbitrary $t$ and choose a sequence $t_n\to +\infty$ such that
$g_{t_n}\to g_t$ (see Remark \ref {accm})  and put
$$R_n= (R_{g_{t_n}(u),g_{t_n}(w')} +  R_{g_{t_n}(u'),g_{t_n}(w)}). $$
Since $\frac{\alpha(t_n)}{t_n}\to 0$, $\Vert R_n\Vert$ and $\Vert R_{g_{t_n}(u'), g_{t_n}(w')}\Vert$ are bounded, and
$$\frac{\alpha(t_n)}{t_n}\geq \left|t_n\Vert R_{g_{t_n}(u'), g_{t_n}(w')}\Vert -\Vert R_n\Vert\right|$$
we must have that both $t_n\Vert R_{g_{t_n}(u'), g_{t_n}(w')}\Vert$ and  $\Vert R_n\Vert$ tend to $0$. In particular, $$0=\lim_{n\to \infty}R_n=R_{g_t(u),g_t(w')} +  R_{g_t(u'),g_t(w)}$$
and
$$0=\lim_{n\to \infty}R_{g_{t_n}(u'), g_{t_n}(w')}=R_{g_{t}(u'), g_{t}(w')}.$$

Then, from (\ref{555}), $$  R_{\tau _t (u),\tau _t (w)}=(\nabla [Y,Z])_{\gamma _v(t)} +C(t)
 $$
Since $(\nabla [Y,Z])_{\gamma _v(t)}$ and $C(t)$ are parallel along $\gamma_v(t)$  one concludes that
$R$ is parallel along $\gamma (t)$ and so $\nabla_v R=0$.
\end{proof}

%\begin {rem}
%In the proof of the above lemma we showed that $R_{\nabla _vY, \nabla _vZ}= 0$ if $\gamma _v(t)$ is a homogeneous geodesic.  The argument used in this proof is not sufficient to prove that the above equality holds for any $v\in \nu _p$ (since  $R_{\nabla _vY, \nabla _vZ}$ is not linear in $v$). Nevertheless, this equality is always true since $\nabla _vY$ and $ \nabla _vZ$ are the initial condition of transvections at $p$ (see \cite [Proposition 3.6]{DOV}) which are bounded fields along $\gamma _v (t)$. Then this equality follows from \cite [Proposition 3.9]{DOV}. {\red No es cierto que $\nabla_v X$ y $\nabla_v Y$ est\'en en $\hat{\nu}_p$ y por lo tanto son los valores iniciales de una transvecci\'on adaptada? Por ende, en particular son valores iniciales de algo en $\mathfrak {tr}^p$ y por  \ref{eq:trp1} R=0?}
%\end{rem}

\begin{rem}\label{2343}
Let $p\in M$ and  let $X$ be a Killing field. Let $B=(\nabla X)_p$ and let $T$ be a tensor. Since the Levi-Civita connection is torsion free one has that
$\nabla_{X_p} T -(\mathcal L _XT)_p = B.T_p$,
where $\mathcal L _X$ is the Lie derivative along $X$ and $B$ acts as a derivation on the algebraic tensor $T_p$. If $T$ is invariant under isometries, or more generally under the flow of $X$, then $\mathcal L _XT= 0$. Hence
$\nabla_{X_p} T  = B.T _p$.
In particular, if $T=R$ is the curvature tensor, then
$\nabla_{X_p} R  = B.R_p$ (cf. \cite [Section 2]{CO}).
Assume now that $M$ is homogeneous with nullity $\nu$  and that $X_p\in \nu _p$.
Then,  by  Lemma \ref {util},
$B.R_p=0$ or, equivalently,
$\mathrm {e}^{tB}(R_p) = R_p$. \qed
\end{rem}
\vspace {.25cm}

We will now introduce some notation that we will keep for the rest of this section.

First, recall that one can identify any   Killing field Y with its initial conditions $(Y) ^p = (Y_p, (\nabla Y)_p)$ at $p$. Under this identification the Lie bracket is given by  (\ref{initb})  and one can think of $\mathcal{K}^G(M)$ as a linear subspace of $T_pM\oplus \mathfrak {so}(T_pM)$.

 On the other hand, recall that if  $g$ be an isometry of $M$ and  $Y$ is a Killing field of $M$, then $g_*(Y)$ is also a Killing
 field, where $(g_*(Y))_q := \mathrm {d}g (Y_{g^{-1}(q)})$.

 Let now $\gamma _v(t)$ be a homogeneous geodesic with $v\in \nu _p$ and let $X\in \mathcal K^G(M)$ be such that
$\gamma _v (t)= \phi _t (p)$, where $\phi _t$ is the flow associated to $X$.

 Let us consider, for   $Z \simeq (u, C) $ in  $\mathcal K ^G(M)\subset T_pM\oplus \mathfrak {so}(T_pM)$, the one parameter family  $Z^t\in \mathcal K ^G(M)$ given by
$$Z^t := (\phi _{-t})_*(Z). $$
The map $Z\to Z^t$ is a one parameter group of automorphisms of the Lie algebra
$\mathcal K ^G(M)$ and, as it is well-known $\frac {\mathrm d\,}{\mathrm dt}_{\vert 0}Z^t =
[X, Z].$

From (\ref {etnabla}) and Lemma \ref {ABC} one has that, if $B=(\nabla X)_p$, then
\begin {equation} \label {23333}
(Z^t)^p = (\mathrm {e}^{-tB}u + t \mathrm {e}^{-tB}Cv, \mathrm {e}^{-tB}C\mathrm {e}^{tB}).
\end{equation}

The following result is crucial for our purposes.

\begin {lemma} \label {mainLe} We keep the previous notation and assumptions.
 Let $$h_t: T_pM\oplus \mathfrak {so}(T_pM)
\to T_pM\oplus \mathfrak {so}(T_pM)$$ be the one parameter group of linear isomorphisms
given by $$h_t(u,C) = (\mathrm {e}^{-tB}u , \mathrm {e}^{-tB}C\mathrm {e}^{tB}).$$ Then $h_t (\mathcal K^G(M)) \subset  \mathcal K^G(M)$ and
$h_t: \mathcal K^G(M) \to \mathcal K^G(M)$ is a Lie algebra automorphism, where
$\mathcal K^G(M)$ is identified with a linear subspace of  $T_pM\oplus \mathfrak {so}(T_pM)$. Moreover, $h_t (\mathcal K_p^G(M))\subset \mathcal K_p^G(M)$ where $\mathcal K_p^G(M)$ is the isotropy algebra at $p$.
\end {lemma}
\begin{proof}
Fix $t\in \mathbb R$ and let $t_k$ be sequence that tends to $+\infty$ and such that
$\mathrm {e}^{-t_kB}$ tends to $\mathrm {e}^{-tB}$ (see Remark \ref {accm}). Let $Z\simeq  (u, C)\in \mathcal K^G(M)$.  Then by (\ref{23333})
$\frac {1}{t_k}(Z^{t_k})^p $ converges to  $ (\mathrm {e}^{-tB}Cv, 0)$. Hence $ (t\mathrm {e}^{-tB}Cv, 0)\in \mathcal K^G(M)$ (observe that this is a transvection at $p$). Then, since $Z^t \in\mathcal K^G(M)$, one obtains that
$$h_t(Z) = Z^t - (t\mathrm {e}^{-tB}Cv, 0)\in \mathcal K^G(M)$$
and hence $h_t (\mathcal K^G(M)) \subset  \mathcal K^G(M)$.
The fact that $[h_t(Y), h_t(Z)]= h_t([Y,Z])$, for any $Y,Z \in \mathcal K^G(M)$,
 follows from a direct computation, using (\ref {initb}) and the fact that
$\mathrm {e}^{-tB}(R_p) = R_p$ (see Lemma \ref{util} and Remark \ref {2343}).

The last assertion is trivial since the isotropy algebra corresponds to those pairs  with first component equals to zero.
\end{proof}

\begin{lemma}\label{main222} We keep the notation and assumptions of this section. Let $X\in \mathcal K^G(M)$ be such that
$\gamma (t)= \phi _t(p)$ is a (homogeneous) geodesic tangent to the the nullity, where $\phi _t$ is the flow associated to $X$.
If $M$ is simply connected, then
there exists a transvection at $p$ with initial conditions $(X_p, 0)$.
\end {lemma}

\begin{proof} We may assume that $G$ is simply connected. In this case the action of $G$ is not necessarily  effective, but almost effective.  Observe that, since $M$ is simply connected, then  $H:=G_p$ must be connected.  Then the  one parameter group of Lie algebra automorphisms $h_t$ induces a one parameter group of automorphisms
$L_t$ of $G$, with $\mathrm {d}_eL_t = h_t$. Moreover, $L_t(H)=H$. Then $L_t$ induces a  one parameter group of diffeomorphisms $\ell _t$ of $M=G/H$ given by
$$\ell _t (gH) = L_t(g)H. $$
Observe that $\ell _t(p)=p$ for all $t\in\mathbb R$. Let $Z=(u,C)\in \mathcal K^G(M)\simeq \mathfrak g$. Then
\begin{equation}\label {for22} \begin{split}
\mathrm {d}_p\ell _t  (u)&=\mathrm {d}_p\ell _t (Z.p) =
\tfrac {\mathrm {d}\,}{\mathrm {d}s}_0\ell_t (\mathrm {Exp} (sZ)p) \\
&=  \tfrac {\mathrm {d}\,}{\mathrm {d}s}_0 L_t(\mathrm{Exp}(sZ))p  =       h_t(Z).p = \mathrm {e}^{-tB}u.
\end{split}
\end{equation}
We conclude that $\mathrm {d}_p\ell _t$ is a linear isometry for all $t\in \mathbb R$. Note that, for all $g\in G$,
$$\ell _t (x) = L_t(g) \ell _t (g^{-1}x)$$ or, equivalently,
$$\ell _t =  {m}_{L_t(g)}\circ \ell _t \circ
 {m}_{g^{-1}}, $$ where $ {m}_r(x) = rx$, $r\in G$,
$x\in M$. Then
$$\mathrm {d}_{gp} \ell _t = \mathrm {d}_p {m}_{L_t(g)}\circ \mathrm {d}_p\ell _t \circ
\mathrm {d}_{gp} m_{g^{-1}}. $$
So $\mathrm {d}_{gp}\ell _t$ is a linear isometry. Since $g\in G$ is arbitrary, then
$\ell _t$ is an isometry that fixes $p$.

From (\ref {for22}) we obtain that the Killing field  associated to the one parameter group of isometries $\ell _t$ is $Y = (0,-B)$.
Then $U= X + Y$ is the transvection with initial conditions $(X_p, 0)$.
\end{proof}

\begin{cor}\label{mainC2} Let $M=G/H$ be a simply connected homogeneous Riemannian manifold with a non-trivial nullity distribution $\nu$. Then, for all $p\in M$,
$v\in \nu _p$ there exists a transvection $Y$ such that $Y_p=v$ (it could happen that  $Y \notin \mathcal K^G(M)$).
\end{cor}

\begin{proof} It follows from Lemma \ref {main222} and the fact that the directions of homogeneous geodesics span $\nu _p$.
\end{proof}

\section {The abelian ideal generated by the transvections} \label {aigt}

Let $M=G/H$  be a locally irreducible Riemannian  homogeneous manifold with a non-trivial nullity.  Let
\begin{equation}\label {eq:trp2}
    \mathfrak {tr}_0^p := \{
    X\in \mathcal K^G(M): X \textrm{ is a transvection at p with } X_p\in \nu_p\}
\end{equation}
(compare with (\ref{eq:trp1})) and let
\begin{equation} \label {eq:trp3}
    \tilde{\mathfrak {tr}}_0^p := \{
    X\in \mathcal K(M): X \textrm{ is a transvection at p with } X_p\in \nu_p\}.  \ \
\end{equation}

By Corollary \ref {mainC2} one has that
$\dim (\tilde{\mathfrak {tr}}_0^p) =\dim (\nu _p)$. Observe that  the inclusion
$\mathfrak {tr}_0^p\subset \tilde{\mathfrak {tr}}_0^p$  could be strict.

\begin{remark}\label{rem33}
The adapted transvections at $p$, with respect to $G$, are given by $[\mathcal K ^G(M), \tilde{\mathfrak {tr}}_0^p]$. In fact, let $v\in \nu _p$, $X\in \tilde{\mathfrak {tr}}_0^p $ with $X_p=v$ and $Z\in \mathcal K^G(M)$. Then $([X,Z])^p = (\nabla _{v}Z, 0)$  which defines an  arbitrary adapted transvection (see (\ref{initb})). Thus, $[\mathcal K ^G(M), \tilde{\mathfrak {tr}}_0^p]\subset\mathcal K ^G(M)$ and hence $\mathcal K ^G(M)+ \tilde{\mathfrak {tr}}_0^p$ is a Lie algebra of Killing fields of $M$. Observe that this Lie algebra does not depend on $p\in M$. This follows from the fact that $\mathrm{Ad}_g(\tilde{\mathfrak {tr}}_0^p)=\tilde{\mathfrak {tr}}_0^{gp}$ for all $g\in G$.
\end{remark}

Throughout the rest of this section we will denote by
 $\mathfrak g = \mathrm {Lie}(G)\simeq \mathcal K^G(M)$ and by  $\tilde{\mathfrak g }= \mathrm {Lie}(I(M))\simeq \mathcal K(M)$. Set
$$\mathfrak g'=\mathfrak g+ \tilde {\mathfrak {tr}}_0^p\subset \tilde{\mathfrak g}.$$
and put
$$\tilde{\mathfrak{tr}}^p=\{X\in \mathfrak g'\,:\, X\ \text{is a transvection at $p$  with }X_p\in \nu^{(1)}_p\}$$
From Remark \ref{rem33}, $\mathfrak g'$ is a Lie subalgebra of $\tilde{\mathfrak g}$.
On the other hand,  from Lemma \ref{propsosculating} and Remark \ref{rem33}, it follows that $\hat{\nu}_p=[\mathfrak g, \tilde{\mathfrak{tr}}^p_0].p$ and form Corollary \ref{mainC2}, $\nu_p=\tilde{\mathfrak{tr}}^p_0.p$. Since $\nu^{(1)}=\nu+\hat{\nu}$ one gets  that
\begin{equation}\label{eq:5.A}
\tilde{	\mathfrak{tr}}^p=[\mathfrak g, \tilde{\mathfrak{tr}^p_0}]+\tilde{\mathfrak{tr}}^p_0
\end{equation}
and
\begin{equation}\label{eq:5.B}\tilde{\mathfrak {tr}}^p.p = \nu ^{(1)}_p.
\end{equation}
The last equality implies, since a transvection at $p$ is defined by its value at $p$, that
(\ref {eq:5.A}) does not depend on $G$ and hence
\begin{equation}\label{eq:5.AA}
	\tilde{\mathfrak {tr}}^p =	[\mathfrak g', \tilde{\mathfrak{tr}^p_0}]+\tilde{\mathfrak{tr}}^p_0=	[\tilde{\mathfrak g}, \tilde{\mathfrak{tr}^p_0}]+\tilde{\mathfrak{tr}}^p_0. \end{equation}

Let
$\mathfrak a$ be the ideal  generated by $\tilde{\mathfrak {tr}}_0^p$ in $\mathfrak g'$.This ideal does not depend on $p\in M$ since any $g\in G$ maps $\nu _p$ into
$\nu _{gp}$ and so it maps  $\tilde{\mathfrak {tr}}_0^p$ into $\tilde{\mathfrak {tr}}_0^{gp}$. Observe that with the same arguments of Theorem \ref{idealmain} and Remark \ref{remcen} one gets that $\mathfrak a $
is an abelian ideal.

 \begin{lemma}\label{lem:notdep} The abelian ideal $\mathfrak a$ does not depend on the presentation group $G$. In particular, $\mathfrak a$ is also the ideal of $\tilde{\mathfrak g}$ generated by $\tilde{\mathfrak{tr}}^p_0$.
\end{lemma}
\begin{proof}
 It suffices to  show that the ideal $\bar {\mathfrak a}$ of $\tilde {\mathfrak g}\simeq \mathcal K(M)$ generated by $\tilde{\mathfrak {tr}}_0^p$ coincides with
$\mathfrak a$.
Let us denote, for a pair of Lie subalgebras $\mathfrak b$ and  $\mathfrak c$ of a Lie algebra,
$\ [\mathfrak c , \mathfrak b]^{0} = \mathfrak b$ and inductively
$[\mathfrak c , \mathfrak b]^{k+1}= [\mathfrak c , [\mathfrak c,\mathfrak b]^k]$. We need to show that for each $r\in \mathbb{N}$,
\begin{equation}\label  {99}   \sum _{k = 0}^r [\mathfrak {g}',  \tilde{\mathfrak{tr}_0}^p]^{k} =\sum _{k =0}^r [\tilde {\mathfrak {g}}, \tilde{\mathfrak {tr}}_0^p]^{k}.
\end{equation}
	 Let us write
\begin{equation}\label {F88}\tilde {\mathfrak g} = \mathfrak g' \oplus \mathfrak h', \end{equation}
where
$\mathfrak h'$ is a linear subspace of the full isotropy algebra $\tilde {\mathfrak g}_p$. 	We shall prove first that
	\begin{equation}\label{eqaux}
		\left[\mathfrak h',\sum_{k=0}^r[\tilde{\mathfrak g},\tilde{\mathfrak {tr}^p_0}]^k\right] \subset \sum_{k=0}^r[\tilde{\mathfrak g},\tilde{\mathfrak {tr}^p_0}]^k\end{equation}
	Since the isotropy at $p$ maps $\nu _p$ into itself and transvections into transvections one has that
	$[\mathfrak h', \tilde{\mathfrak {tr}_0^p}]\subset \tilde{\mathfrak {tr}_0^p}$ and so (\ref{eqaux}) holds for $r=0$. Assume now that (\ref{eqaux}) is true for some $r$. Then
\begin{eqnarray*}	\left[\mathfrak h',[\tilde{\mathfrak g},\tilde{\mathfrak {tr}^p_0}]^{r+1}\right]&= &
	 	\left[\mathfrak h',[\tilde{\mathfrak g},[\tilde{\mathfrak g},\tilde{\mathfrak {tr}^p_0}]^{r}]\right]=\left[[\mathfrak h',\tilde{\mathfrak g}],[\tilde{\mathfrak g},\tilde{\mathfrak {tr}^p_0}]^{r}]\right]+\left[\tilde{\mathfrak g}, [\mathfrak h',[\tilde{\mathfrak g},\tilde{\mathfrak {tr}^p_0}]^r]\right]\\
	 	&\subset & [\tilde{\mathfrak g},[\tilde{\mathfrak g},\tilde{\mathfrak{tr}^p_0}]^r]+\left[\tilde{\mathfrak g},\sum_{k=0}^r[\tilde{\mathfrak g},\tilde{\mathfrak{tr}^p_0}]^k\right]\subset \sum_{k=0}^{r+1}[\tilde{\mathfrak g},\tilde{\mathfrak{tr}^p_0}]^k.
	 	\end{eqnarray*}
 This implies that (\ref{eqaux}) holds for every $r\in \mathbb{N}$.
 	
Let us turn to the proof of (\ref{99}). It trivially holds for $r=0$, and 	from equation (\ref{eq:5.AA}), one gets that it holds for $r=1$.
So assume that (\ref{99}) is true for some $r$.  In particular,
$$[\tilde{\mathfrak g},\tilde{\mathfrak{tr}^p_0}]^r\subset \sum _{k =0}^r [\tilde {\mathfrak {g}}, \tilde{\mathfrak {tr}}_0^p]^{k} = \sum _{k = 0}^r [\mathfrak {g}',  \tilde{\mathfrak{tr}_0}^p]^{k}$$ From this and from (\ref{eqaux}) one gets that
$$[\tilde{\mathfrak g},\tilde{\mathfrak{tr}^p_0}]^{r+1}=[\mathfrak g'+\mathfrak h',[\tilde{\mathfrak g},\tilde{\mathfrak{tr}^p_0}]^r] \subset  \sum _{k = 0}^{r+1}[\mathfrak {g}',  \tilde{\mathfrak{tr}_0}^p]^{k} $$
which implies the non-trivial inclusion on (\ref{99}).
\end{proof}

\vspace{.2cm}

\begin{proof} [Proof of Theorem \ref{dakl}] Part (1) is just Corollary \ref{mainC2}. Part (2) follows from Remark \ref{rem33}. Part (3)  follows from Lemma \ref{lem:notdep} and its preceding paragraph.
	
	Let us prove (4): by (\ref{nupic}) and (\ref{nupic4}) one has that $\nu ^{(1)} = \nu + \hat {\nu}$ and $$\nu ^{(2)} = \nu ^{(1)} + \hat{\hat {\nu}}$$
		where $\hat{\hat {\nu}}_q= \mathrm {span}\,  \{\nabla _v Z: Z\in \mathcal K ^G (M) , v\in \hat {\nu} _q \}$.
		
	Observe that, from (\ref{eq:5.B}), $\nu ^{(1)}_p=\tilde{\mathfrak {tr}}^p.p $ and from (\ref{eq:5.AA}), $\tilde{\mathfrak {tr}}^p\subset \mathfrak a$. So $\nu ^{(1)}_p \subset \mathfrak a . p$.
	On the other hand, from Lemma \ref{propsosculating} and remark \ref{rem33}, any element  $v\in \hat{\nu}_p$ is the initial value of a transvection $X\in [\mathfrak g,\tilde{\mathfrak{tr}}^p_0]$. Then for each $Z\in \mathfrak g\simeq \mathcal{K}^G(M)$ one has that $\nabla_v Z=\nabla_{X_p}Z=[X,Z]_p$. Hence $\hat {\hat{\nu}}_p\subset [\mathfrak g ,[\mathfrak g,\tilde{\mathfrak {tr}}^p_0]].p\subset \mathfrak a.p$  and so $\nu ^{(2)}_p\subset
	\mathfrak a .p$. This proves the first assertion of (4).

 \noindent  Let $a\in \bar {A}$. One has that
 $\tilde{\mathfrak {tr}}^{ap}= \mathrm{Ad}_a(\tilde{\mathfrak {tr}}^{p})$ (resp., $\tilde {\mathfrak {tr}}_0^{ap}= \mathrm{Ad}_a
 (\tilde{\mathfrak {tr}}_0^{p})$).  Then, since $\bar{A}$ is abelian, $\tilde{\mathfrak {tr}}^{ap}= \tilde{\mathfrak {tr}}^{p}$ (resp.,$\tilde{\mathfrak {tr}}_0^{ap}= \tilde{\mathfrak {tr}}_0^{p}$) for all $a\in \bar {A}$. Observe that the distribution
 $\nu ^{(1)}$ (resp., $\nu $), restricted to $\bar{A}\cdot p$, is spanned by the elements $X$ of $\tilde{\mathfrak{tr}}^{p}$ (resp., $\tilde{\mathfrak{tr}}_0^{p}$). Since $\nabla X=0$ when restricted to $\bar {A}.p$, we obtain the second part of (4).
  \end{proof}
\vspace{.25cm}

 \begin{proof}[Proof of Corollary \ref{cor:A}]

The isotropy $H$ is trivial by Theorem A (3) of \cite {DOV}.  Observe that, with the same argument,  any presentation group of $M$ must have trivial isotropy. So, in this case, $\mathfrak g=\mathfrak g'=\tilde{\mathfrak g}$.

Let $\mathfrak a$ be  the abelian ideal of $\mathfrak g$
generated by $\tilde{\mathfrak{tr}}^p_0$ (see part (3) of Theorem \ref{dakl}). Since $M$ is non-flat, $\mathfrak a .p$ is properly contained in $T_pM$. From Theorem A (1) of \cite {DOV} one has that $\nu ^{(2)}_p$ has codimension $1$ in $T_pM$ (cf. \ref{osculatingandbounded}). Then, from part (4) of Theorem \ref{dakl},
$\mathfrak a$ has codimension $1$ in $\mathfrak g$. If $w\in \mathfrak g$ does not belong to $\mathfrak a$ then
$\mathfrak g = \mathbb R w \rtimes \mathfrak a$. Let $A$ be the Lie normal subgroup of $G$ associated to
$\mathfrak a$ and $L$ be the one-dimensional Lie subgroup of $G$ associated to $\mathbb Rw$. Let $\pi :\tilde A \simeq \mathbb R^{n-1}\to A$ and $\pi ':\tilde L \simeq \mathbb R \to L$  be the universal cover of $A$ and $L$, respectively. Then the map $m: \tilde L \times \tilde A \to G$ given by $m((l,a)) = \pi ' (l)\pi (a)$ is a covering map. Since $G\simeq M$ is simply connected one has that $m$ is a diffeomorphism. Hence $L\simeq \mathbb R$ and $A\simeq \mathbb R^{n-1}$ and thus $G$ is isomorphic   to  a semidirect product $\mathbb R ^{n-1}\rtimes \mathbb R$.
\end{proof}
\vspace{.25cm}

We finish this section with a lemma and a remark that will be needed in the following section.
\begin{lemma}\label{nois}. Let $M=G/H$ be a (connected) homogeneous Riemannian manifold
where $G$ is a connected closed Lie subgroup of $I^o(M)$. Let $N$  be a connected closed abelian normal subgroup of $G$ and let $N\cdot p \subset M$ be any  orbit. Then  the isotropy $N_p$ is trivial.
\end{lemma}
\begin{proof}
Since $N$ is abelian, then $N=T^k\times \mathbb R^m$ where $T^k$ is a torus.
Since $N$ is a normal subgroup and $G$ is connected the conjugation by $g\in G$
$I_g: N\to N$ acts trivially on $T^k$. Thus, $T^k$ is included in the center $C(G)$ of $G$. Since $N_p$ is compact, $N_p\subset T^k$.   Then $N_p\subset C(G)$ and hence $N_p= \{1\}$ since $G$ acts effectively.
\end{proof}

\begin{rem}\label{nois2} We have the following infinitesimal version of Lemma \ref{nois}. Let $\mathfrak a$ be an abelian ideal of  $\mathfrak g = \mathrm {Lie}(G)$, where $G$ acts, by isometries, almost effectively and transitively on $M$. Then $\mathfrak a \cap \mathfrak g_p = \{0\}$ for all $p\in M$, where $\mathfrak g _p$ is the isotropy algebra at $p$. This can be proved as follows: the Killing form of $\mathfrak g$ is null when restricted to $\mathfrak a$ and it is negative definite when restricted to the isotropy algebra (see the proof of Lemma 2.7 of \cite {DOV}).
\end{rem}

\

\section{Non-solvable examples}\label{nonsolvable}

Let $K$ be a  simple simply connected  compact Lie group with Lie algebra $\mathfrak k$ and let   $\rho: K \to SO_n$  be an irreducible orthogonal  representation. Let us consider the semidirect product $G= \mathbb {R}^n\rtimes_\rho K$. Namely
$G= \mathbb {R}^n\times K$ and
$(v,k)(v',k')= (v +\rho (k) (v'), kk') $. Observe that the Lie algebra of $G$ is $\mathfrak g = \mathbb R^n\rtimes_{\rho} \mathfrak k$, which is the vector space $\mathbb{R}^n\oplus \mathfrak k$ with the following Lie bracket $[\,,\,]$: if $v,w\in \mathbb{R}^n$ and $X,Y\in \mathfrak k$ then
$$[v,w]=0, \ \ [X,Y]=[X,Y]_{\mathfrak k},\ \ [X,v]=d\rho_e(X)(v).$$ Let us consider the simply connected Riemannian manifold $M=G$, with a left invariant metric which restricted to $\mathbb{R}^n$ coincides with the canonical metric of $\mathbb{R}^n$. Then for each $w\in \mathbb{R}^n\subset\mathfrak g$, the Killing vector field of $M$ induced by $w$ is $$\tilde{w}_{(v,k)}=(w,0_k). $$ Moreover, from (\ref{feq}), and the fact that $\rho$ is an orthogonal representation, one obtains that  for each $v,w\in \mathbb{R}^n$
 \begin{equation}\label{parallelrn}\nabla_{\tilde{v}}\tilde{w}=0.\end{equation}

Let $\mathcal D$ be the (integrable) distribution of $M$ defined at $p=(v,k)$ by  $\mathcal D_p=  (\mathbb{R}^n,0_k)$. It is not hard to prove that $\mathcal{D}$ is left-invariant, and from equation (\ref{parallelrn}) one obtains that $\mathcal D$ is an autoparallel distribution.

\begin{rem}\label{idealabeliano} The only proper ideal of $\mathfrak g$ is $\mathbb{R}^n$. In fact, consider the projection $\pi:\mathfrak g\to \mathfrak g/\mathbb R^{n}\simeq \mathfrak k$ and let $\mathfrak J$ be a proper ideal of $\mathfrak g$. Then $\pi(\mathfrak J)$ is an ideal of this quotient, and since $\mathfrak k$ is simple, either $\pi(\mathfrak J)=\{0\}$ or $\pi(\mathfrak J)\simeq \mathfrak k$.
	If  $\pi(\mathfrak J)=0$, then $\mathfrak J\subset \mathbb{R}^n$. But then $\mathfrak J$ is a $\mathfrak k$-invariant subspace of $\R^n$. Since the action of $K$ on $\mathbb{R}^n$ is irreducible one concludes that  $\mathfrak J=\mathbb{R}^n$.
	A similar argument shows that if $\pi(\mathfrak J)=\simeq \mathfrak k$, then $\mathfrak J=\mathfrak g$.
	\end{rem}

\begin{rem} \label{obvio} Let $u\in \mathfrak g$ and let $\tilde u$ be the Killing field induced by $u$, i.e. $\tilde u _q= u.q$. Then, if $g$ is an isometry, $g_*(\tilde u) = \textrm{Ad}_g(u)^{\tilde \,}$. In  particular, $u$ is a transvection at $q$ if and only if $g_*(\tilde u)$ is a transvection at $g(q)$.
 \end{rem}

\begin{lemma}\label{nse}
If $M$ has a non-trivial Euclidean de Rham factor then $\mathcal D$ is a parallel distribution and the Killing fields induced by $\mathbb R^n
\simeq \mathrm {Lie}(\mathbb R^n)$ are parallel fields. Furthermore, $M$ is isometric to the Riemannian product of the Euclidean space $\mathbb R^n$ by the Lie group $K$ with a left invariant metric.

\end{lemma}
\begin{proof}
Let $M= E\times M'$ where $E\simeq \mathbb R^r$ is the Euclidean de Rham factor and $M'$ is the product of the irreducible de Rham factors. Then the space of Killing fields of $M$ decomposes as
\begin{equation}\label{desckilling}
\mathcal K(M) = \mathcal K(\mathbb R^r)\oplus
\mathcal K(M')
\end{equation}
and the full isometry group of $M$ decomposes as
\begin{equation}
\label{desciso} I(M)=I(E)\times I(M').
\end{equation}
Let $\mathcal E$ be the parallel distribution of $M$ associated to the flat de Rham factor. Then its first osculating distribution
$\mathcal E^{(1)}$ (see equation \ref{firstosculating}) coincides with $\mathcal E$ and so
\begin{equation}\label{7}
\nabla _{\mathcal E_p}X \subset \mathcal E_p
\end{equation} for all
$X\in \mathcal K^G(M)\simeq \mathfrak g = \mathrm {Lie}(G)$, $p\in M$. The same argument used in \cite [Proposition 3.6] {DOV} shows that there exists always a transvection induced by $G$ in the direction of $\nabla _vX$, for all $X\in\mathcal K^G(M)$, $v\in \mathcal E_p$ (the argument used in the proof of such a proposition does not use that there is no flat de Rham factor, except for finding a transvection  not in the nullity).

Case {(\it i)}: $\nabla _vX= 0$, for all $X\in\mathcal K^G(M)$, $v\in \mathcal E_p$. Then, if  $X\in \mathcal K^G(M)$, $X= T + Z$ where
$T$ is a transvection of the Euclidean factor and $Z$ is a Killing field of $M'$ (by enlarging $G$ to the full isometry group and with the usual identification \ref{desckilling}). Then, if $X_p\perp \mathcal E_p$, for some $p\in M$, the field $X$ must be always tangent to the distribution
$\mathcal E^\perp$ associated to the factor $M'$.
Let us regard $K\subset I(M)$ and consider the projection  $\pi: K \to I(E)$ according to the decomposition \ref{desciso}. Then $\pi (K)$
is a compact subgroup of isometries of $E\simeq \mathbb R^r$. Then
there exists $q\in E$ which is fixed by $\pi(K)$ or equivalently, since $\pi (K)$ is connected, $u.q= 0$ for all $u\in \mathrm {Lie}(\pi (K))$. We may regard $E$ as a fixed integral manifold of $\mathcal E$ and so we will regard $q$ as an element of $M$. Then
$\mathfrak k .q \perp \mathcal E_q$ where $\mathfrak k =\mathrm {Lie}(K)$. Then, from what previously observed, any Killing field $X$ induced by an element of $\mathfrak k$ lies in $\mathcal E^\perp$. Since $\mathcal E^\perp$ is $G$-invariant, the Killing fields induced by $G$ which lie in this distribution are an ideal, let us say $\mathfrak I$, of $\mathfrak g$. Since $\mathfrak k \subset \mathfrak I$, from remark \ref{idealabeliano} we conclude that $\mathfrak I = \mathfrak g$ and thus $\mathcal E=\{0\}$. A contradiction that show that this case cannot hold if $E\neq\{0\}$.

Case {(\it ii)}: $\nabla _vX\neq 0$, for some $X\in \mathfrak g $ and some $v\in \mathcal E _p$. Since, by (\ref {7}), $\nabla _vX\in \mathcal E_p$ we  conclude that there exists a non-trivial transvection $Y$ induced by $G$ such that $Y_p\in \mathcal E_p$. Such a transvection must be always tangent to the Euclidean factor of $M$ and must be a transvection at any point. Then, for any $g\in G$, $g_*(Y)$ lies in $\mathcal E$ and it is a transvection at any point of $M$  (see Remark \ref{obvio}).
Then the ideal $\mathfrak J$ of $\mathfrak g$ generated by $Y$ consists of transvections at any point that lie in $\mathcal E$. Since such transvections must commute we conclude that $\mathfrak J$ is a non-trivial abelian ideal of $\mathfrak g$. Thus by Remark \ref{idealabeliano}, $\mathfrak J = \mathbb{R}^n$. In particular $r=\dim(\mathcal E)=n$.   Observe that in this case any element $v$ of $\mathfrak J$ induces a parallel vector field $\tilde v$ of $M$ which proves the first assertions.
This implies that $\mathcal D= \mathcal E$.

\noindent Let us  consider again the natural projection  $\pi: K\to I(E)$. Then, as in the proof of case (i), $\mathfrak k .q\perp \mathcal E_q= \mathcal D_q$. Then, since $G$ acts simply transitively on $M$ we must have that $K$ act simply transitively on the integral manifold by $q$ of the parallel distribution associated to the factor $M'$ of $M$. From this we conclude that $M'$ is isometric to $K$ with a left invariant metric.
\end{proof}

\begin{lemma}\label{nse2} Let $M= \mathbb R^n \rtimes _\rho K$ be a semidirect product with a left invariant metric  which restricted to $\mathbb{R}^n$ coincides with the canonical metric of $\mathbb{R}^n$, where $\rho :  K \to \mathbb SO_n$ is an irreducible orthogonal representation and $K$ is a simply connected compact simple Lie group with $\dim (K)<n$. Assume that $M$ has no Euclidean de Rham factor. Then $M$ is an irreducible Riemannian manifold.
\end{lemma}

\begin{proof} Let $\mathfrak g = \mathbb R^n \rtimes_{\mathrm {d}\rho} \mathfrak k$ be the Lie algebra of $G =\mathbb R^n\rtimes _\rho K$, let $M_i$ be an irreducible de Rham factor of $M$ and
let $\pi_i: \mathfrak g \to \mathrm {Lie}(I(M_i))$ the the natural projection. Observe that $\mathfrak J:=\ker (\pi_i)$ is an  ideal which is  properly contained in   $\mathfrak g$. Then from Remark \ref{idealabeliano} either $\mathfrak J=\{0\}$ or $\mathfrak J =\mathbb R^n$.

Assume that $\mathfrak J=\mathbb R^n$. Let us denote also by $\pi _i$ the induced Lie group morphism from $G$ into $I(M_i)^o$. Then $\pi _i (G)=\pi_i(K)$ is compact. Since $\pi_i(G)$ must act transitively on $M_i$, we conclude that $M_i$ is compact.

If $\mathfrak J =\{0\}$ then $\pi _i (\mathbb R^n)$ is an $n$-dimensional ideal of
$\pi_i(\mathfrak g)$. By Remark \ref {nois2} $\dim (\pi _i (\mathbb R^n). q)=n$ for any $q\in M_i$ and hence
$\dim (M_i) \geq n$. Since $\dim (M)= n +\dim (K) < 2n$ there is at most one of such factors with $\mathrm {ker} (\pi _i)= \{0\}$. Note that there must be one  of such factors, let us say $M_1$, since $M$ in non-compact.

Let $\mathcal D$ be the parallel foliation of $M$ associated to $M_1$. Let  $M'$ be the product of the de Rham factors $M_i$ of $M$ different from $M_1$ (and so $\mathrm {ker}(\pi_i) =\mathbb R^n$). Assume $M'$ in non-trivial.
 We identify $M'$ with the integral manifold of
$\mathcal D^\perp$ by $e = (0,1)$.  We have shown that  any irreducible de Rham factor of $M'$ is compact,  and so $M'$ is compact.  Let $\pi : G\to I(M')^o$ be the natural projection. Then clearly $\ker(\pi)=\mathbb{R}^n$ and  so, $\pi (G)= \pi (K)$ and $\pi (G)$ is locally isomorphic to $K$. Since $\mathcal D^\perp$ is $G$-invariant we have that $M'=H\cdot e$ for some  simply connected compact Lie subgroup $H$ of $K$. Observe, since $G$ acts without isotropy, that $H$ that has no isotropy and in particular acts effectively on $M'$.
 Observe that $\pi(G)$ contains the restriction to $M'$ of  $H$ (that we identify with $H$ acting by left multiplications on $M'$).
  From \cite{OT} (see also \cite{Go}) one has that
$H$ is a normal subgroup of $I(M)^o$. Then $H$  is a normal subgroup of
$\pi (G)$. Since $H\neq \{1\}$ and $K$ is simple, we conclude that  $H\simeq \pi (G)$. Then $\dim (M') = \dim (K)$ and so $\dim (M_1) = n$. Observe, since
$\pi (\mathbb R^n)=0$, that the Killing fields induced by the elements in
$\mathrm {Lie}(\mathbb R^n)  \simeq \mathbb R^n\subset \mathfrak g$ lie in the distribution $\mathcal D$. Then, by Remark \ref{nois2}, $\mathbb R^n$ acts transitively (by isometries) on $M_1$. Then $M_1$ is flat. A contradiction that proves the $M'$ is trivial.
\end{proof}

\begin {prop}\label{mex}
Let $K$ be a simply connected compact simple Lie group and let
$\rho : K\to SO_n$ be an irreducible orthogonal representation. Let $\mathbb V_0$ be a non-trivial vector subspace of $\mathbb R^n$ such that $\dim (\mathbb V_0)(1 +\dim (K))< n$. Then there exists a left invariant metric $\langle \, ,\, \rangle$ on $G=\mathbb R^n \rtimes_{\rho} K$
such that $M = (G,\langle \, ,\, \rangle)$ is an irreducible Riemannian manifold and the nullity distribution $\nu$ of $M$ at $(0,e)$ contains $\mathbb V_0$ (hence the index of nullity of $M$ is at least
$\dim (\mathbb V_0)$).
\end{prop}
\begin{proof}

Let us define, for $i\geq 1$,
$$\mathbb V_{i} = [\mathfrak g, \mathbb V_{i-1}] +\mathbb V_{i-1}
=  [\mathfrak k, \mathbb V_{i-1}] +\mathbb V_{i-1}.$$
Since $[\mathfrak k,\mathbb V_i]\subset \mathbb V_{i+1}$ and $K$ acts irreducibly on $\mathbb R^n$ there exists $d\in \mathbb N$ such that
$\mathbb V_{d}=\mathbb R^n$. We choose $d$  to be the minimal with this property. Observe that $\dim (\mathbb V_1)\leq    \dim(\mathbb V_0)\dim (K) + \dim(\mathbb V_0)
= \dim (\mathbb V_0)(1+\dim (K))<n$. So $\mathbb V_1$ is a proper subspace of $\mathbb R^n$  and thus $d\geq 2$. Let $e= (0,1)\in G$, let $(\, , \, )$ be the canonical inner product of $\mathbb R^n$ and let $\langle \, , \,\rangle '$ be an inner product in $\mathfrak k \simeq \mathfrak k.e$. Let $w\in \mathfrak k$ and $v\in \mathbb V_{d-1}$ be such that $ [w,v]\notin \mathbb V_{d-1}$. Let $0\neq z\in \mathfrak k$ be perpendicular to  $w$.  Let $v'$ be the orthogonal projection of $[w,v]$ into the orthogonal complement $\mathbb V_{d-1}^\perp$ of $\mathbb V_{d-1}$ in $\mathbb V_d = \mathbb R^n$ (any element of $\mathbb R^n\subset \mathfrak g$ defines a Killing field which is identified with its value at $e$). Let $\bar {\mathbb V}$ be the orthogonal complement of $v'$ in $\mathbb V_{d-1}^\perp$. So $$T_eG=\mathbb V_{d-1}\oplus \overline{\mathbb V}\oplus \mathbb R v'\oplus \mathfrak k.$$
Let us consider the inner product $ \langle \, ,\, \rangle $ on $T_eG $ defined by:

1) $ \langle \, ,\, \rangle  =
(\, , \, )\times \langle \, , \,\rangle '$
when restricted to  $\big( \mathbb V_{d-1}\oplus \bar {\mathbb V}\big) \oplus \mathfrak k$.

2) $ \langle \,  , \, \rangle  = (\, , \, )$ when restricted to $\mathbb R^n\subset T_eG$.

3) $v'$ is perpendicular to the orthogonal complement of the linear span of
$w,\, z$ in $\mathfrak k$ and
$ \langle v' , w\rangle  = a$,
 $ \langle v' , z \rangle  = b$,
 where
$a, b$ are small generic constants, that will be fixed later, and such that
 $ \langle \, , \, \rangle $
is a positive definite inner product of $T_eG$.

From formula (\ref {feq}) one obtains that the elements of $\mathbb V_{d-2}$ define transvections at $e$ (by making use that the bracket by an element of $\mathfrak k$ defines a skew symmetric transformation of $\mathbb R^n$).

Let us write $[v, z] = \lambda v' + u$, where $u\perp  v'$.
Identifying $x\in \mathfrak g$ with the associated Killing field $q\mapsto x.q$ we have that
\begin {equation}
\begin{split}
2\langle \nabla _{w.e}v , z.e\rangle &= \langle [w,v].e,z.e \rangle
+ \langle [w,z].e,v.e \rangle + \langle [v, z].e, w.e \rangle\\
&= \langle [w,v].e,z.e \rangle +
 \langle [v, z].e, w.e \rangle = b + \lambda a
 \end{split}
 \end {equation}
 (observe that $[v,w]\in \mathfrak k$ and $v\in \mathbb{V}_{d-1}$ are perpendicular).
One can choose the   generic constants $a$ and $b$ such that $b+\lambda a\neq 0$. Then the Killing field
 $q\mapsto v.q$ is not a transvection at $e$. Then, by Lemma \ref {nse},  $M$ has no Euclidean factor. Hence, by Lemma \ref {nse2}, $M$ is an irreducible Riemannian manifolds.
 From formula (\ref {feqR}) we obtain that
  $\mathbb V_{d-2}. e \subset \nu _e$ since
the elements of both $\mathbb V_{d-2}$ and
$[\mathfrak g,\mathbb V_{d-2}]\subset \mathbb V_{d-1}$ induce transvections.
\end{proof}

\

\begin{proof}[Proof of Theorem \ref{mex2}]
If follows directly from Proposition \ref{mex}.
\end{proof}

%\begin {cor}\label {nosol}
%For any simply connected compact simple Lie group $K$ there exist examples of simply connected irreducible homogeneous Riemannian manifolds  with a non-trivial nullity which  are diffeomorphic to the product of  $K$ by a Euclidean space of arbitrary large dimension  (and ).
%\end{cor}

%%% TERMINE de revisar.

%\begin{prop}\label {closure}
%Let $M=G/H$ be a homogeneous Riemannian
%manifold  and $\mathcal D$ be $G$-invariant
%integrable distribution. Then
%\begin{enumerate}
%\item The closure of the integral manifolds of
%$\mathcal D$ define a fibration
%$\mathcal F$ on $M$.
%\item If the local  quotient by the integral
%manifolds of $\mathcal D$ is a Riemannian
%submersion, then $M\to M/\mathcal F$ is a
%Riemannian submersion.
%\end{enumerate}
%\end{prop}

%\begin{lemma}\label {indepre} Let $M=G/H$ be
%a locally irreducible Riemannian manifold with a non-trivial nullity.
%\end{lemma}

\vspace{1cm}

\noindent
\begin{tabular}{l|cl|cl}
Antonio J. Di Scala & & Carlos E. Olmos  & & Francisco Vittone\\
\footnotesize Dipartimento di Scienze Matematiche&  & \footnotesize FaMAF, CIEM-Conicet & &\footnotesize DM, ECEN, FCEIA - Conicet\\
\footnotesize Politecnico di Torino &  &\footnotesize Universidad Nac. de C\'ordoba& &\footnotesize Universidad Nac. de Rosario \\
 \footnotesize Corso Duca degli Abruzzi, 24& &\footnotesize Ciudad Universitaria  & &\footnotesize Av. Pellegrini 250\\
 \footnotesize 10129, Torino, Italy  & & \footnotesize 5000, C\'ordoba, Argentina & &\footnotesize 2000, Rosario, Argentina \\
\footnotesize{antonio.discala@polito.it} & & \footnotesize{olmos@famaf.unr.edu.ar} & & \footnotesize{vittone@fceia.unr.edu.ar}\\
\footnotesize{http://calvino.polito.it/$\sim$adiscala}& &
 & &\footnotesize{www.fceia.unr.edu.ar/$\sim$vittone}
\end{tabular}

\end{document}